\documentclass[11pt]{amsart}
\usepackage[margin=1in]{geometry}
\usepackage{amsthm, amsmath,amsfonts,amssymb,euscript,hyperref,graphics,color,slashed,mathrsfs}
\usepackage{graphicx}
\usepackage{comment}
\usepackage{import}
\usepackage{tikz}
\usepackage{latexsym}
\usepackage{mathtools}

\newtheorem{theorem}{Theorem}[section]
\newtheorem{lemma}[theorem]{Lemma}
\newtheorem{proposition}[theorem]{Proposition}

\newtheorem{definition}[theorem]{Definition}

\numberwithin{equation}{section}

\begin{document}
\title {A note on the number of irrational odd zeta values}

\author{Li LAI, Pin YU}

\address{Department of Mathematics and Yau Mathematical Sciences Center, Tsinghua University\\ Beijing, China}
\email{lail14@mails.tsinghua.edu.cn}
\email{yupin@mail.tsinghua.edu.cn}
%\classification{11J72 (primary), 11M06, 33C20 (secondary)}
%\keywords{irrationality, zeta values, hypergeometric series.}
\maketitle

\begin{abstract}
It is proved that, for all odd integer $s \geqslant s_0(\varepsilon)$, there are at least $\big( c_0 - \varepsilon \big) \frac{s^{1/2}}{(\log s)^{1/2}} $ many irrational numbers among the following odd zeta values: $\zeta(3),\zeta(5),\zeta(7),\cdots,\zeta(s)$. The constant $\displaystyle c_0 = 1.192507\ldots$ can be expressed in closed form.

The work is based on the previous work of Fischler, Sprang and Zudilin \cite{FSZ2019}, improves the lower bound $2^{(1-\varepsilon)\frac{\log s}{\log\log s}}$ therein. The main new ingredient is an optimal design for the zeros of the auxiliary rational functions, which relates to the inverse of Euler totient funtion.
\end{abstract}

\tableofcontents

\section{Introduction}
\label{Chapter_Introduction}

The Riemann zeta function $\zeta(s)$ is one of the most fascinating objects in mathematics. Due to the work of Euler and Lindemann, it is well known that for any positive integer $k$, the Riemann zeta value $\zeta(2k)$ is a (non-zero) rational multiple of $\pi^{2k}$, therefore, is transcendental. One may want to further investigate the odd zeta values, i.e., the numbers $\zeta(2k+1)$'s. It is conjectured that $\pi, \zeta(3), \zeta(5),\zeta(7),\cdots$ are algebraically independent over $\mathbb{Q}$, but very little is known.

We mention a few works on this subject. In 1978, Ap\'ery proved that $\zeta(3)$ is irrational \cite{Ap79}. (see also van der Poorten's report \cite{Po79} and Beuker's alternative proof \cite{Be79}. For a survey, see \cite{Fi04}.) In 2000, Ball and Rivoal \cite{BR01} (see also Rivoal \cite{Ri00}) showed that for all odd integer $s \geqslant 3$, we have the following asymptotics as $s \rightarrow +\infty$:
\[\dim_{\mathbb{Q}}\Big({\rm{Span}}_{\mathbb{Q}}\big(1,\zeta(3),\zeta(5),\zeta(7),\cdots,\zeta(s)\big) \Big) \geqslant \frac{1 + o(1)}{1 + \log 2} \log s.\]
The proof of Ball and Rivoal makes use of Nesterenko's linear independence criterion \cite{Ne85} and the following auxiliary rational functions:
\[R_{n}^{\rm{(BR)}}(t) = n!^{s-2r} \frac{\prod_{j=0}^{(2r+1)n}(t-rn+j)}{\prod_{j=0}^{n}(t+j)^{s+1}}.\]
%A punch line of the argument is that $\sum_{t=1}^{\infty} R_n^{\rm{(BR)}}(t)$ is a $\mathbb{Q}$-linear combination of $1$, $\zeta(3)$, $\zeta(5)$, $\cdots$, $\zeta(s-2)$ and $\zeta(s)$.
As a corollary, there are infinitely many irrational numbers among odd zeta values. In 2018, Zudilin \cite{Zu18} studied the following rational functions (with $s = 25$)
\[R_n^{\rm{(Z)}}(t) = 2^{6n} n!^{s-5} \frac{\prod_{j=0}^{6n} (t-n + j/2)}{ \prod_{j=0}^{n} (t+j)^{s+1}},\]
and proved that both series $\sum_{t=1}^{\infty} R_n^{\rm{(Z)}}(t)$ and $\sum_{t=1}^{\infty} R_n^{\rm{(Z)}}\left(t+\dfrac{1}{2}\right)$ are $\mathbb{Q}$-linear combinations of odd zeta values with related coefficients, it provides a new elimination procedure. Zudilin's new idea inspires many works afterwards (see Sprang \cite{Sp18} and Fischler \cite{Fi19}). Based on further developments in \cite{Zu18} and an important arithmetic observation of Sprang \cite[Lemma 1.4]{Sp18}, In 2018, Fischler, Sprang, and Zudilin \cite{FSZ2019} proved for all $\varepsilon >0$, for all odd integer $s$ which is sufficiently large with respect to $\varepsilon$, with the help of the following rational functions:
\[R_{n}^{\rm(FSZ)}(t) = D^{3Dn} n!^{s+1-3D} \frac{ \prod_{j=0}^{3Dn}(t-n+j/D)}{ \prod_{j=0}^{n}(t+j)^{s+1}},\]
the number of irrationals in the set $\big\{\zeta(3),\zeta(5),\zeta(7),\cdots,\zeta(s)\big\}$ is at least $2^{(1-\varepsilon)\frac{\log s}{\log \log s}}$. (see also \cite{FSZ18})

\smallskip

In the current work, we prove the following result:

\smallskip

\begin{theorem}\label{MainThm}
For any small $\varepsilon > 0$, for all odd integer $s$ sufficiently large with respect to $\varepsilon$, there are at least
\[ \left(c_0 - \varepsilon\right) \frac{s^{\frac{1}{2}}}{\log^{\frac{1}{2}} s} \]
many irrational numbers among $\big\{\zeta(3),\zeta(5),\zeta(7),\cdots,\zeta(s)\big\}$, where the constant
\[ c_0 = \sqrt{\frac{4\zeta(2)\zeta(3)}{\zeta(6)} \left( 1 - \log \frac{\sqrt{4e^2+1}-1}{2} \right) }  = 1.192507\ldots. \]
\end{theorem}

\smallskip

The proof is a natural extension to the original ideas of Zudilin \cite{Zu18} and Sprang
\cite[Lemma 1.4]{Sp18}. Our main strategies are exactly the same as \cite{FSZ2019}, though an amount of small technical modifications are involved. The major new ingredient of our work is an optimal design for the rational zeros of the auxiliary rational functions, this design is in connection with the inverse totient problem.

The structure of this note is as follows: In Section \ref{Auxiliary_functions_and_linear_forms}, we introduce the auxiliary rational functions $R_n(t)$ and related linear forms. In Section \ref{Arithmetic_lemmas}, we study the arithmetic of the denominators appeared in the linear forms. In Section \ref{Analysis_lemmas}, we bound the growth of the linear forms. In section \ref{Elimination_procedure}, we prove Theorem \ref{MainThm}. Finally, in the last section \ref{Remarks}, we show that under certain constraints, the FSZ auxiliary functions constructed in this note are the most economical ones.
\bigskip

\section{Auxiliary functions and linear forms}
\label{Auxiliary_functions_and_linear_forms}

Let $r = \frac{{\rm num}{(r)}}{{\rm den}{(r)}}$ be a positive rational number, where ${\rm num}{(r)}$ and ${\rm den}{(r)}$ are the numerator and denominator of $r$ in reduced form, respectively. We refer to $r$ the Ball-Rivoal length parameter \cite{BR01}. Eventually we will take the rational number $r$ arbitrarily close to $r_0 = \frac{\sqrt{4e^2+1}-1}{2} \approx 2.26388$ in order to maximize certain quantity.

Let $s$ be a positive odd integer and $B$ be a positive real number. We will always assume that
\begin{itemize}
\item[(1)] Both $s$ and $B$ are larger than some absolute constant.
\item[(2)] $s \geqslant 10(2r+1)B^2$.
\end{itemize}
Eventually we will take $B = cs^{1/2}/\log^{1/2} s$ for some constant $c$.

\smallskip
\begin{definition}\label{SetDefinition} We define the following two sets which depend only on $B$:
\begin{enumerate}
\item[(1)] the \emph{Denominator Set}
\[ \mathcal{\varPsi}_B = \{ b \in \mathbb{N} \mid \varphi(b) \leqslant B \},\]
where $\varphi(\cdot)$ is the Euler totient function.
\item[(2)] the \emph{Zero Set}
\[ \mathcal{F}_{B} = \left\{ \frac{a}{b} \in \mathbb{Q} ~\bigg|~ b \in \mathcal{\varPsi}_B, 1 \leqslant a \leqslant b, \text{~and~} \gcd(a,b) = 1 \right\} .\]
\end{enumerate}
\end{definition}
\smallskip

The zero set $\mathcal{F}_B$ consists of the zeros in the interval $(0,1]$ of our auxiliary rational functions, and the denominator set $\mathcal{\varPsi}_B$ consists of different denominators of the zeros. We collect some properties for these two sets. The first property is known in the topic about inverse totient problem.

\smallskip
\begin{proposition}\label{SetProposition}
\begin{enumerate}
\item [(1)] The size of the set $\mathcal{\varPsi}_B$ is
\[ \left| \mathcal{\varPsi}_B \right| = \left( \frac{\zeta(2)\zeta(3)}{\zeta(6)} + o_{B \rightarrow +\infty}(1) \right) B. \]
\item [(2)] For any $b \in \mathcal{\varPsi}_B$, we have
\[ \left\{ \frac{1}{b},\frac{2}{b},\cdots,\frac{b}{b}  \right\} \subset \mathcal{F}_{B}.  \]
\item [(3)] If $B$ is larger than some absolute constant, then
\[ \left| \mathcal{F}_{B} \right| \leqslant B^2. \]
\end{enumerate}
\end{proposition}

\begin{proof}
For the first proposition, we refer the readers to \cite{Dr70} or \cite{Ba72}. Since $|\mathcal{F}_{B}|$ $=$ $\sum_{b \in \mathcal{\varPsi}_B} \varphi(b)$, a summation by parts argument gives $|\mathcal{F}_{B}| =\left(\frac{1}{2}\frac{\zeta(2)\zeta(3)}{\zeta(6)}+o(1)\right)B^2$. Now, $\frac{\zeta(2)\zeta(3)}{\zeta(6)} = 1.94\ldots < 2$, the third proposition follows. For the second proposition, note that if $b \in \mathcal{\varPsi}_B$ and $b'$ is any divisor of $b$, then $\varphi(b') \leqslant \varphi(b) \leqslant B$, so $b' \in \mathcal{\varPsi}_B$. Therefore, for any $k \in \{ 1,2,\cdots,b \}$, we have $\frac{k}{b} = \frac{k/\gcd(k,b)}{b/\gcd(k,b)} \in \mathcal{F}_{B}$.
\end{proof}
\smallskip

We define the integer
\begin{equation}\label{P}
P_{B,{\rm den}(r)} = 2{\rm den}(r) \cdot \text{LCM}_{\substack{b \in \mathcal{\varPsi}_B \\ p \mid b}}\{ p-1 \} ,
\end{equation}
where $\text{LCM}$ means taking the least common multiple. As a convention, the letter $p$ always denotes prime numbers.

For given $r$, $s$ and $B$, we define the following auxiliary rational functions.

\smallskip
\begin{definition}[FSZ constructions]\label{auxiliary_function}
For any positive integer $n$ which is a multiple of $P_{B,{\rm den}(r)}$, we define the rational function
\[R_n(t) = A_1(B)^n A_2(B)^n \frac{n!^{s+1}}{\left( \frac{n}{{\rm den}(r)} \right)!~^{{\rm den}(r)(2r+1)|\mathcal{F}_B|}} \frac{(t-rn)\prod_{\theta \in \mathcal{F}_{B}} \prod_{j=0}^{(2r+1)n-1} \left(  t-rn+j+\theta \right) }{\prod_{j=0}^{n}(t+j)^{s+1}},\]
where
\[ A_1(B) = \prod_{b \in \mathcal{\varPsi}_B} b^{(2r+1)\varphi(b)}, \]
we refer to $A_1(B)^n$ the \emph{major arithmetic (wasting) factor}, and
\[ A_2(B) = \prod_{b \in \mathcal{\varPsi}_B} \prod_{p \mid b} p^{\frac{(2r+1)\varphi(b)}{p-1}} ,\]
we refer to $A_2(B)^n$ the \emph{minor arithmetic (wasting) factor}.
\end{definition}
\smallskip

Notice that by \eqref{P}, both $A_1(B)^n$ and $A_2(B)^n$ are integers, also, $\frac{n}{{\rm den}(r)}$, $rn$, and $(2r+1)n$ are integers. In the following lemma we estimate $A_1(B)$ and $A_2(B)$:

\smallskip
\begin{lemma}\label{A1A2}
We have
\[ A_1(B) = \exp\left( \left( \frac{1}{2} \frac{\zeta(2)\zeta(3)}{\zeta(6)} + o_{B \rightarrow + \infty}(1) \right)\left(2r+1\right) B^2 \log B \right), \]
and for any $B$ larger than some absolute constant,
\[ A_2(B) \leqslant \exp\left( 10(2r+1)B^2(\log\log B)^2 \right) .\]
\end{lemma}

\begin{proof}
We start by $\log A_1(B) = (2r+1)\sum_{b \in \mathcal{\varPsi}_B} \varphi(b)\log b$. Firstly,
\begin{align*}
\log A_1(B) &\geqslant (2r+1)\sum_{b \in \mathcal{\varPsi}_B} \varphi(b)\log \varphi(b) \\
            &= (2r+1) \int_{1^{-}}^{B} x\log x ~{\rm d} | \mathcal{\varPsi}_x|,
\end{align*}
an integration by parts argument with the fact $| \mathcal{\varPsi}_x| = \left(\frac{\zeta(2)\zeta(3)}{\zeta(6)}+o_{x \rightarrow +\infty}(1)\right) x$ (see Proposition \ref{SetProposition} $(1)$) gives $\log A_1(B) \geqslant \left(2r+1\right)\left( \frac{1}{2} \frac{\zeta(2)\zeta(3)}{\zeta(6)} + o_{B \rightarrow + \infty}(1) \right) B^2 \log B$. On the other hand, it is well known (see, for instance, \cite[Thm 2.9]{MV06}) that
\[ \varphi(m) \geqslant \left( e^{-\gamma} + o_{m \rightarrow +\infty}(1) \right)\frac{m}{\log\log m}, \]
where $\gamma = 0.577\ldots$ is Euler's constant. For any $b \in \mathcal{\varPsi}_B$, since $\varphi(b) \leqslant B$, we derive that
\begin{equation} \label{b}
b \leqslant ( e^{\gamma}+o_{B \rightarrow +\infty}(1) ) B \log\log B,
\end{equation}
thus $\log A_1(B) \leqslant (2r+1)(1+o_{B\rightarrow +\infty}(1))\log B \sum_{b \in \mathcal{\varPsi}_B} \varphi(b)$, a summation by parts argument as above gives $\log A_1(B) \leqslant \left(2r+1\right)\left( \frac{1}{2} \frac{\zeta(2)\zeta(3)}{\zeta(6)} + o_{B \rightarrow + \infty}(1) \right) B^2 \log B$. Combining the two parts we obtain the estimate for $A_1(B)$.

Now, for $A_2(B)$, by \eqref{b} and $e^{\gamma} = 1.78\ldots < 2$, when $B$ is larger than some absolute constant, we have
\begin{align*}
\log A_2(B) &\leq (2r+1) \sum_{b \leqslant 2B\log\log B} \varphi(b) \sum_{p \mid b} \frac{\log p}{p-1} \\
&= (2r+1) \sum_{p \leqslant 2B\log\log B} \frac{\log p}{p-1} \sum_{\substack{b \leqslant 2B\log\log B \\ p \mid b}} \varphi(b)
\end{align*}
Since $\sum_{\substack{b \leqslant 2B\log\log B \\ p \mid b}} \varphi(b) \leqslant \sum_{\substack{b \leqslant 2B\log\log B \\ p \mid b}} b \leqslant \frac{4B^2 (\log\log B)^2}{p}$, it implies that
\[ \log A_2(B) \leqslant 4(2r+1) B^2 (\log\log B)^2 \sum_{p} \frac{\log p}{p(p-1)}  ,\]
the estimate for $A_2(B)$ follows.
\end{proof}
\smallskip

We proceed to construct linear forms in Hurwitz zeta values. Since the numerator and denominator of $R_{n}(t)$ have a common factor $\prod_{j=0}^{n} (t+j)$, it can be rewritten as $R_{n}(t) = \frac{Q_{n}(t)}{\prod_{j=0}^{n} (t+j)^{s}}$, where $Q_{n}(t)$ is a polynomial in $t$ with rational coefficients. Since $\deg R_n <0$ (see below), we know that $R_n(t)$ has a (unique) partial fraction expansion
\begin{equation}\label{a_ik}
R_n(t) = \sum_{i=1}^{s}\sum_{k=0}^{n} \frac{a_{i,k}}{(t+k)^i}
\end{equation}
with coefficients $a_{i,k} \in \mathbb{Q}$. Note that these coefficients $a_{i,k}$ also depend on $n$, $r$, $s$, and $B$.

We list two properties of $R_n(t)$ and $a_{i,k}$ which will be used later in Lemma \ref{LemmaLin}:
\begin{itemize}
\item[(1)] As a rational function, the degree of $R_{n}(t)$ is
\[ \deg R_n = 1 + (2r+1)|\mathcal{F}_B|n - (s+1)(n+1)  \leqslant -2.\]
This is due to $|\mathcal{F}_B| \leqslant B^2$ and $s \geqslant 10(2r+1)B^2$.
\item[(2)] The auxiliary function $R_{n}(t)$ has the following symmetry:
\[R_{n}(-t-n) = - R_{n}(t).\]
In view of the fact that $(2r+1)n$ is even and $s$ is odd, the proof is a straightforward computation. In particular, since the partial fractional expansion for $R_n(t)$ is unique, we derive that
\[(-1)^i a_{i,k} =- a_{i,n-k},\]
for all $1\leqslant i\leqslant s$, $0 \leqslant k \leqslant n$.
\end{itemize}

For all $\theta \in \mathcal{F}_B$, we define the following quantities:
\begin{equation}\label{r_n_theta}
r_{n,\theta} = \sum_{m=1}^{\infty} R_{n}(m+\theta).
\end{equation}
The notations $r_{n,\theta}$ are adopted to keep pace with those in \cite{FSZ2019}. There is no risk to be confused with the Ball-Rivoal length parameter $r$.

We recall the definition of the Hurwitz zeta values:
\[\zeta(i,\alpha)  = \sum_{m=0}^{\infty} \frac{1}{(m+\alpha)^i},\]
where $i\geqslant 2$ is an integer and $\alpha$ is a positive real number.

The following lemma is the same as \cite[Lemma 1]{FSZ2019} (for a proof, see therein):

\smallskip
\begin{lemma}[linear forms]\label{LemmaLin} For all $\theta \in \mathcal{F}_B$, we have
\[r_{n,\theta} = \rho_{0,\theta} + \sum_{\substack{3 \leqslant i \leqslant s \\ i ~ {\rm odd}}} \rho_{i} \zeta\left(i,\theta \right),\]
where the rational coefficient
\[\rho_{i} = \sum_{k=0}^{n} a_{i,k} \quad \text{for $3 \leqslant i \leqslant s$, $i$ odd,}\]
does not depend on $\theta \in \mathcal{F}_B$, and
\[\rho_{0,\theta} = - \sum_{k=0}^{n}\sum_{\ell=0}^{k}\sum_{i=1}^{s} \frac{a_{i,k}}{\left( \ell+\theta \right)^i}.\]
\end{lemma}
\bigskip

\section{Arithmetic lemmas}
\label{Arithmetic_lemmas}

The following proposition is elementary, we omit the proof:

\smallskip
\begin{proposition}\label{consecutive_prod}
Let $L \in \mathbb{N} \cup \{ 0 \}$. Suppose $x_1,x_2,\cdots,x_L$ be any $L$ consecutive terms in an integer arithmetic progression with common difference $b \in \mathbb{N}$, then for any prime $q \nmid b$, we have
\[ v_q(x_1x_2\cdots x_L) \geq \sum_{i=1}^{\infty} \left\lfloor \frac{L}{q^i} \right\rfloor .\]
In the degenerate case of $L=0$, we view $x_1x_2\cdots x_L = 1$.
\end{proposition}
\smallskip

For any $\frac{a}{b} \in \mathcal{F}_{B}$ with $\gcd(a,b)=1$, we define the following polynomials:
\begin{align}
F_{b,a}(t) &= \frac{\prod_{p \mid b} p^{\frac{(2r+1)n}{p-1}}}{\left( \frac{n}{{\rm den}(r)} \right)!~^{{\rm den}(r)(2r+1)}}\cdot b^{(2r+1)n} \prod_{j=0}^{(2r+1)n -1} \left( t-rn + j + \frac{a}{b} \right) \nonumber \\
&= \frac{\prod_{p \mid b} p^{\frac{(2r+1)n}{p-1}}}{\left( \frac{n}{{\rm den}(r)} \right)!~^{{\rm den}(r)(2r+1)}} \prod_{j=0}^{(2r+1)n -1} \left( bt-brn + a + bj \right). \label{F_ba}
\end{align}
Then we define
\[ \widetilde{F}_{b,a}(t) = \begin{cases} F_{b,a}(t) &\text{if~} \frac{a}{b} \neq 1, \\ (t-rn)F_{1,1}(t) &\text{if~} \frac{a}{b} = 1. \end{cases} \]
Notice that since $n \in P_{B,{\rm den}(r)}\mathbb{N}$, by \eqref{P}, all of $\frac{(2r+1)n}{p-1}$, $\frac{n}{{\rm den}(r)}$, $rn$, and $(2r+1)n$ are integers. By Definition \ref{auxiliary_function}, we have
\begin{equation}\label{R_equalsto_FFF}
R_n(t) = n!^{s+1} \frac{\prod_{\frac{a}{b} \in \mathcal{F}_{B}}\widetilde{F}_{b,a}(t)}{\prod_{j=0}^{n}(t+j)^{s+1}}.
\end{equation}

For a formal series $U(t) = \sum_{\ell=0}^{\infty} u_{\ell} t^{\ell} \in \mathbb{Q}[[t]]$, we denote by $[t^{\ell}](U(t))$ the $\ell$-th coefficient of $U(t)$, i.e., $[t^{\ell}](U(t)) = u_{\ell}$.

As usual, we denote by $d_n=\text{LCM}\{1,2,\cdots,n\}$ the least common multiple of the first $n$ positive integers. By the prime number theorem, we have $\lim_{n \rightarrow +\infty} d_n^{1/n} = e$. We first establish the following arithmetic property of $\widetilde{F}_{b,a}(t)$:

\smallskip
\begin{proposition}\label{AriF}
For any nonnegative integers $\ell$ and $k$, we have
\[d_{n}^{\ell} \cdot [t^{\ell}](\widetilde{F}_{b,a}(t-k)) \in \mathbb{Z} \]
\end{proposition}

\begin{proof}
Note that we only need to prove the proposition with $\widetilde{F}_{b,a}$ replaced by $F_{b,a}$. If $\ell > \deg F_{b,a} =(2r+1)n$, the proposition trivially holds. In the rest of the proof, we assume $\ell \leq \deg F_{b,a}$.

For a prime $q \mid b$, the $q$-adic order of the factor $\frac{\prod_{p \mid b} p^{\frac{(2r+1)n}{p-1}}}{\left( \frac{n}{{\rm den}(r)} \right)!~^{{\rm den}(r)(2r+1)}}$ is nonnegative, so by \eqref{F_ba}, the $q$-adic order of every coefficient of $F_{b,a}(t-k)$ is nonnegative. Therefore, for any prime $q \mid b$, we have $v_q( d_{n}^{\ell} \cdot [t^{\ell}](F_{b,a}(t-k))) \geqslant 0$.

Now consider a prime $q \nmid b$. Notice that $[t^{\ell}]\left(\prod_{j=0}^{(2r+1)n -1} \left( b(t-k)-brn + a + bj \right)\right)$ is a sum of finitely many terms all of the form
\begin{equation}\label{piecewise_consecutive_prod}
 b^{\ell} \prod_{i=1}^{\ell + 1}\prod_{j \in J_i} \left( -bk-brn + a + bj \right),
\end{equation}
where $J_i$ is a set consisting of $L_i \in \mathbb{N} \cup \{ 0 \}$ consecutive integers such that $L_1 + L_2 + \cdots + L_{\ell + 1} = (2r+1)n - \ell $. By Proposition \ref{consecutive_prod}, we derive that the $q$-adic order of the expression \eqref{piecewise_consecutive_prod} is
\[ v_q(\eqref{piecewise_consecutive_prod}) \geqslant  \sum_{i=1}^{\infty} \sum_{j=1}^{\ell + 1} \left\lfloor \frac{L_j}{q^i} \right\rfloor. \]
For a fixed $i \geqslant 1$, we have $\sum_{j=1}^{\ell + 1} \left\lfloor \frac{L_j}{q^i} \right\rfloor \geqslant \sum_{j=1}^{\ell + 1} \frac{L_j - (q^i - 1)}{q^i} = \frac{(2r+1)n+1}{q^i} - \ell - 1 > \left\lfloor \frac{(2r+1)n}{q^i} \right\rfloor - \ell - 1$, but the left hand side is a nonnegative integer, so we obtain that $\sum_{j=1}^{\ell + 1} \left\lfloor \frac{L_j}{q^i} \right\rfloor \geqslant \max(0, \left\lfloor \frac{(2r+1)n}{q^i} \right\rfloor - \ell)$. Therefore,
\begin{align}
v_q(\eqref{piecewise_consecutive_prod}) &\geqslant \sum_{i=1}^{\lfloor \log_q n \rfloor} \left( \left\lfloor \frac{(2r+1)n}{q^i} \right\rfloor - \ell \right) \nonumber \\
&\geqslant \sum_{i=1}^{\lfloor \log_q n \rfloor} \left( {\rm den}(r)(2r+1) \left\lfloor \frac{n/{\rm den}(r)}{q^i} \right\rfloor - \ell \right) \label{v_q} \\
&= v_q\left( \left( \frac{n}{{\rm den}(r)} \right)!~^{{\rm den}(r)(2r+1)} \right) - \ell v_q(d_n). \nonumber
\end{align}
(The non-trivial part is for cases $q\leqslant n$, for $q>n$, the above derivation is also valid but degenerates to trivial results.) In conclusion, for any prime $q \nmid b$, by equation \eqref{F_ba} and equality \eqref{v_q}, we find that $d_{n}^{\ell} \cdot [t^{\ell}](F_{b,a}(t-k))$ is a sum of finitely many terms, each of these terms has nonnegative $q$-adic order, this completes the proof of Proposition \ref{AriF}.
\end{proof}

We prove the following arithmetic lemma, which corresponds to \cite[Lemma 2]{FSZ2019}. In our situation, the Ball-Rivoal length parameter $r$ is just a rational number (not necessary an integer or a half integer), so we have to modify the proof for $d_n^{s+1-i} a_{i,k} \in \mathbb{Z}$, but the rest of proof is the same as \cite[Lemma 2]{FSZ2019}.

\smallskip
\begin{lemma}[arithmetic lemma]\label{LemmaAri}
We have
\[d_n^{s+1-i} \rho_{i} \in \mathbb{Z}\]
for all odd integers $i$ with $3 \leqslant i \leqslant s$, and we have
\[d_{n+1}^{s+1} \rho_{0,\theta} \in \mathbb{Z}\]
for all $\theta \in \mathcal{F}_B$.
\end{lemma}

\begin{proof}
For any $k \in \{ 0,1,\cdots,n \}$ and any $i \in \{ 1,2,\cdots,s \}$, by comparing \eqref{R_equalsto_FFF} with the partial fraction expansion \eqref{a_ik} of $R_n(t)$, and by viewing $t^{s+1}R_n(t-k) \in \mathbb{Q}[[t]]$ as a formal series, we have
\begin{align*}
a_{i,k} &= [t^{s+1-i}]\left( t^{s+1}R_{n}(t-k) \right) \\
&= \binom{n}{k}^{s+1} [t^{s+1-i}]\left( \prod_{\frac{a}{b} \in \mathcal{F}_{B}} \widetilde{F}_{b,a}(t-k) \prod_{\substack{0 \leqslant j \leqslant n \\ j \neq k}} \left(1+ \frac{t}{j-k} \right)^{-s-1} \right) \\
&= \binom{n}{k}^{s+1} \sum_{\substack{\underline{\ell} \\ {\rm sum}(\underline{\ell}) = s+1-i}} \prod_{\frac{a}{b} \in \mathcal{F}_{B}} [t^{\ell_{b,a}}]\left( \widetilde{F}_{b,a}(t-k) \right) \prod_{\substack{0 \leqslant j \leqslant n \\ j \neq k}} \frac{(-1)^{\ell_j}\binom{s+\ell_j}{\ell_j}}{(j-k)^{\ell_j}}
\end{align*}
where the sum is taken for all tuples $\underline{\ell}$ consisting of nonnegative integers $\ell_{b,a}$ and $\ell_j$ such that
\[ {\rm sum}(\underline{\ell}) = \sum_{\frac{a}{b} \in \mathcal{F}_{B}} \ell_{b,a} + \sum_{\substack{0 \leqslant j \leqslant n \\ j \neq k}} \ell_j = s+1-i . \]
By Proposition \ref{AriF} and the fact $d_n^{\ell_j} \frac{1}{(j-k)^{\ell_j}} \in \mathbb{Z}$, we derive that
\[ d_{n}^{s+1-i} a_{i,k} \in \mathbb{Z}. \]

Once $d_{n}^{s+1-i} a_{i,k} \in \mathbb{Z}$ is established, the rest of the proof is the same as \cite[Lemma 2]{FSZ2019}. We only mention that the most remarkable part is $d_{n+1}^{s+1} \rho_{0,\theta} \in \mathbb{Z}$, which is proved by showing
\[ \sum_{i=1}^{s} \frac{d_{n+1}^{s+1} a_{i,k}}{(\ell + \theta)^{i}} \]
is an integer for any $0 \leqslant \ell \leqslant k \leqslant n$ and $\theta \in \mathcal{F}_B$. It uses the fact that $R_n(t)$ has zeros $-n+\theta, -n+1+\theta, -n+2 +\theta ,\cdots, \theta$ for $\theta \in \mathcal{F}_B \setminus \{1\}$, this observation origins from Sprang \cite[Lemma 1.4]{Sp18}.

\end{proof}
\bigskip

\section{Analysis lemmas}
\label{Analysis_lemmas}

Under our assumptions $s \geqslant  10(2r+1)B^2$ and $B \gg 1$, we have the following:

\smallskip
\begin{lemma}[analysis lemma]\label{LemmaAna}
We have
\[\lim_{n \rightarrow +\infty} \left(r_{n,1}\right)^{\frac{1}{n}} = g(x_0) ,\]
where
\[g(X) = A_1(B)A_2(B) {\rm den}(r)^{(2r+1)|\mathcal{F}_B|} (X+2r+1)^{(2r+1)|\mathcal{F}_B|} \left( \frac{(X+r)^r}{(X+r+1)^{r+1}} \right)^{s+1},\]
and $x_0$ is the unique positive real solution of the equation
\[f(X) = \left( \frac{X+2r+1}{X} \right)^{|\mathcal{F}_B|} \left( \frac{X+r}{X+r+1} \right)^{s+1} = 1.\]
Moreover, for any $\theta \in \mathcal{F}_B$, we have
\[\lim_{n \rightarrow +\infty} \frac{r_{n,1}}{r_{n,\theta}} = 1.\]
\end{lemma}
\smallskip

Before proving the analysis lemma, we first collect some properties of the functions $f$ and $g$. Note that these two functions depend only on $r,s$ and $B$.

\smallskip
\begin{proposition}\label{g_x_0}
Let $f(x)$ and $g(x)$ be the functions in Lemma \ref{LemmaAna} (defined on $x \in (0,+\infty)$). Then
\begin{enumerate}
\item[(1)] There exists a unique $x_0 \in (0,+\infty)$ such that $f(x_0) = 1$, $f(x) > 1$ on $(0,x_0)$ and $f(x) < 1$ on $x \in (x_0,+\infty)$. Moreover,
    \[ x_0 < \frac{r(r+1)|\mathcal{F}_B|}{s+1 - (2r+1)|\mathcal{F}_B|} .\]
\item[(2)] If we fix $r \in \mathbb{Q}_{+}$ and assume in addition that $B = cs^{1/2}/\log^{1/2} s$ for some positive constant $c$, when $s \rightarrow +\infty$, we have
    \[ \lim_{\substack{s \rightarrow +\infty \\ B = cs^{1/2}/\log^{1/2} s}} g(x_0)^{\frac{1}{s+1}} = \exp\left( \frac{\zeta(2)\zeta(3)}{4\zeta(6)}(2r+1)c^2 \right) \frac{r^r}{(r+1)^{r+1}} ,\]
\end{enumerate}
\end{proposition}

\begin{proof}
For the first proposition, by calculating $\frac{f'(x)}{f(x)}$, we find that $f'(x) = 0$ has a unique positive solution $x_1$ which satisfies
\begin{equation}\label{x_1}
(s+1-(2r+1)|\mathcal{F}_B|)x_1^2 + (2r+1)(s+1-(2r+1)|\mathcal{F}_B|)x_1 - r(r+1)(2r+1)|\mathcal{F}_B| = 0,
\end{equation}
and $f$ is decreasing on $(0,x_1)$, increasing on $(x_1,+\infty)$. Since $f(0^{+}) = +\infty$ and $f(+\infty) = 1$, there exists a unique $x_0$ satisfying all the requirements. The last (very weak) bound for $x_0$ comes from $x_0 < x_1$ and \eqref{x_1}.

The second proposition follows from the estimates for $A_1(B)$, $A_2(B)$, $|\mathcal{F}_B|$ (see Proposition \ref{SetProposition} and Lemma \ref{A1A2}), and $x_0 \rightarrow 0$.
\end{proof}
\smallskip

Now we prove Lemma \ref{LemmaAna}. We claim that it can be proved by the same strategy in \cite[Lemma 3]{FSZ2019}, but we give a slightly modified proof.

\smallskip
\begin{proof}[proof of Lemma \ref{LemmaAna}]
For any $\theta \in \mathcal{F}_B$, since $R_n(m+\theta) = 0 $ for $m=1,2,\cdots,rn-1$, we define the shift version of the auxiliary rational functions:
\[ \widehat{R}_{n}(t) = R_{n}(t+rn) ,\]
then by \eqref{r_n_theta} we have
\begin{equation}\label{r_ntheta}
r_{n,\theta}  = \sum_{k=0}^{\infty} \widehat{R}_{n}(k+\theta).
\end{equation}

We have the following two expressions for $\widehat{R}_{n}(t)$:
\begin{align}
\widehat{R}_{n}(t) &= A_1(B)^n A_2(B)^n \frac{n!^{s+1}}{\left( \frac{n}{{\rm den}(r)} \right)!~^{{\rm den}(r)(2r+1)|\mathcal{F}_B|}} \cdot \frac{ t ~ \prod_{\theta' \in \mathcal{F}_B} \prod_{j=0}^{(2r+1)n-1} \left(t+j+\theta' \right)}{ \prod_{j=0}^{n} (t+rn+j)^{s+1}} \label{Shift_R_Expression1} \\
&= A_1(B)^n A_2(B)^n \frac{n!^{s+1}}{\left( \frac{n}{{\rm den}(r)} \right)!~^{{\rm den}(r)(2r+1)|\mathcal{F}_B|}}  \nonumber\\
&\qquad \times t \cdot \left( \prod_{\theta' \in \mathcal{F}_B } \frac{\Gamma(t+(2r+1)n+\theta')}{\Gamma(t+\theta')} \right) \cdot  \left( \frac{\Gamma(t+rn)}{\Gamma(t+(r+1)n+1)} \right)^{s+1} \label{Shift_R_Expression2}.
\end{align}

We define $c_1 = \min(e^{-10s/r},\frac{x_0}{10})$, which is independent of $n$. To estimate the series \eqref{r_ntheta} for $r_{n,\theta}$, we divide it into three parts:
\[r_{n,\theta} = \left( \sum_{0 \leqslant k < c_1n} + \sum_{c_1 n \leqslant k \leqslant n^{10}} + \sum_{k>n^{10}} \right) \left( \widehat{R}_{n}(k+\theta) \right).\]

For the first part, by \eqref{Shift_R_Expression1}, a direct computation and trivial estimates give that $ \frac{\widehat{R}_{n}^{\prime}(t)}{\widehat{R}_{n}(t)} > 0$ for all $t \in(0,2c_1 n]$. So $\widehat{R}_{n}(t)$ is increasing on $t \in (0,2c_1 n]$, we have
\begin{equation}\label{firstpart}
\sum_{0 \leqslant k < c_1n} \widehat{R}_{n}(k+\theta) < \left( c_{1}n+1 \right)\widehat{R}_{n}([c_{1}n]+\theta).
\end{equation}

To deal with the middle part, for all $c_1 n \leqslant k \leqslant n^{10}$, we denote by $\kappa = \kappa(k,n) = \dfrac{k}{n} \in [c_1, +\infty)$. By applying Stirling's formula in the weak form
\[ \Gamma(x) = x^{O_{x \rightarrow +\infty}(1)}\left( \frac{x}{e} \right)^x \]
for the equation \eqref{Shift_R_Expression2}, a calculation shows that as $n \rightarrow +\infty$:
\begin{align}
\widehat{R}_{n}(k+\theta) &= n^{O(1)} \cdot A_1(B)^n A_2(B)^n {\rm den}(r)^{(2r+1)|\mathcal{F}_B|n}  \nonumber \\
&\quad \times \left( \frac{(\kappa + 2r + 1)^{\kappa + 2r + 1}}{\kappa^\kappa} \right)^{|\mathcal{F}_B|n} \left( \frac{(\kappa + r)^{\kappa + r}}{(\kappa + r + 1)^{\kappa + r + 1}} \right)^{(s+1)n}  \nonumber \\
&=   n^{O(1)} \cdot \left(  f(\kappa)^{\kappa} g(\kappa)  \right)^n \nonumber \\
&= n^{O(1)} \cdot h(\kappa)^n  \label{secondpartoriginal}
\end{align}
uniformly for any $k \in[c_1n,n^{10}]$ and any $\theta \in \mathcal{F}_B$ (the absolute bound for $O(1)$ depends only on $s,B,r$ and ${\rm den}(r)$). Where the function $h(x)$ is defined for $x>0$ as $h(x)= f(x)^x g(x)$, a direct computation shows that $\frac{h'(x)}{h(x)} = \log f(x)$. Hence, $h(x)$ achieves its maximum only at $x = x_0$ with maximal value $h(x_0) = g(x_0)$.

In particular, we have the following bound for each $k \in  [c_1 n , n^{10}]$:
\begin{equation}\label{secondpart}
\widehat{R}_{n}(k+\theta) \leqslant  n^{O(1)} \cdot g(x_0)^n.
\end{equation}

Finally, for the tail part, for any $k > n^{10}$, when $n \geq \max(10(2r+1), \frac{10A_1(B)A_2(B)}{g(x_0)})$, by \eqref{Shift_R_Expression1} with some trivial estimates and our assumption $s \geq 10(2r+1)B^2$, we have
\begin{align*}
\widehat{R}_{n}(k+\theta) &<   \frac{(2A_1(B)A_2(B)n)^{(s+1)n}}{k^{\frac{9}{10}(s+1)n + 2}} \\
&< \left( \frac{g(x_0)}{2} \right)^n \frac{1}{k^2}.
\end{align*}
As a conclusion, we obtain the following bound for the tail part for all sufficiently large $n$:
\begin{equation}\label{thirdpart}
\sum_{k> n^{10}}  \widehat{R}_{n}(k+\theta) \leqslant \left( \frac{g(x_0)}{2} \right)^n.
\end{equation}

\smallskip

Now, in view of  the estimates \eqref{firstpart}, \eqref{secondpart} and \eqref{thirdpart}, we have
$r_{n,1} \leqslant n^{O(1)} g(x_0)^n$. On the other hand \eqref{secondpartoriginal} implies that $r_{n,1} \geqslant \widehat{R}_{n}(\lfloor x_0n \rfloor) = n^{O(1)}h\left(x_0 + o(1)\right)^n$. Therefore,
\[\lim_{n \rightarrow +\infty} \left(r_{n,1}\right)^{\frac{1}{n}} = g(x_0).\]

To prove the last statement in the lemma, we first fix an arbitrary (sufficiently) small $\varepsilon_0 > 0$. For all $\theta \in \mathcal{F}_B$, we have
\begin{equation}\label{r_n_geq}
r_{n,\theta} \geqslant \sum_{(x_0 - \varepsilon_0)n \leqslant k \leqslant (x_0 + \varepsilon_0)n} \widehat{R}_{n}(k+\theta).
\end{equation}
In view of the estimates \eqref{firstpart}, \eqref{secondpartoriginal} and \eqref{thirdpart}, we also have
\begin{align}
r_{n,\theta} &\leqslant n^{O(1)}\max\big( h(x_0 - \varepsilon_0), h(x_0 + \varepsilon_0) \big)^n \quad + \quad  \sum_{(x_0 - \varepsilon_0)n \leqslant k \leqslant (x_0 + \varepsilon_0)n} \widehat{R}_{n}(k+\theta) \nonumber\\
&< (1+\varepsilon_0) \sum_{(x_0 - \varepsilon_0)n \leqslant k \leqslant (x_0 + \varepsilon_0)n} \widehat{R}_{n}(k+\theta), \label{r_n_leq}
\end{align}
provided $n$ is sufficiently large with respect to $\varepsilon_0$, i.e., $n \geqslant n_0(\varepsilon_0)$. For all $k$ with $(x_0 - \varepsilon_0)n \leqslant k \leqslant (x_0 + \varepsilon_0)n$, let $\kappa = \kappa(n,k) = \dfrac{k}{n}$ as before. We now use the fact that, for any fixed real number $\tau$,
\begin{equation}\label{Gamma_shift_ratio}
\dfrac{\Gamma(x+\tau)}{\Gamma(x)} = \left( 1+o_{x \rightarrow +\infty}(1) \right) x^{\tau} .
\end{equation}
Applying \eqref{Gamma_shift_ratio} to \eqref{Shift_R_Expression2}, we derive that
\begin{align}
\frac{\widehat{R}_{n}(k+1)}{\widehat{R}_{n}(k+\theta)} &= \left( 1+o(1) \right) \cdot \left( \frac{\kappa + 2r + 1}{\kappa} \right)^{|\mathcal{F}_B| (1-\theta)}\left( \frac{\kappa + r}{\kappa + r+1} \right)^{(s+1)(1-\theta)} \nonumber \\
&= \left( 1+o(1) \right) \cdot f(\kappa)^{1 - \theta} \label{R_x0_pmvarepsilon_n}
\end{align}
uniformly for $k \in [(x_0 - \varepsilon_0)n , (x_0 + \varepsilon_0)n]$ as $n \rightarrow +\infty$. By \eqref{r_n_geq}, \eqref{r_n_leq} and \eqref{R_x0_pmvarepsilon_n} we find that
\[\left( 1+o(1) \right)\frac{1}{1+\varepsilon_0} f(x_0 + \varepsilon_0)^{1 - \theta} \leqslant \frac{r_{n,1}}{r_{n,\theta}} \leqslant \left( 1+o(1) \right)(1+\varepsilon_0)f(x_0 - \varepsilon_0)^{1-\theta},\]
thus
\[\frac{1}{1+\varepsilon_0} f(x_0 + \varepsilon_0)^{1 - \theta} \leqslant \liminf_{n \rightarrow +\infty} \frac{r_{n,1}}{r_{n,\theta}} \leqslant \limsup_{n \rightarrow +\infty} \frac{r_{n,1}}{r_{n,\theta}} \leqslant  (1+\varepsilon_0) f(x_0 - \varepsilon_0)^{1-\theta}.\]
It is true for all sufficiently small $\varepsilon_0>0$. Letting $\varepsilon_0 \rightarrow 0^+$, we deduce that
\[\lim_{n \rightarrow +\infty} \frac{r_{n,1}}{r_{n,\theta}} = 1.\]
This completes the proof of Lemma \ref{LemmaAna}.
\end{proof}
\bigskip

\section{Elimination procedure and proof of the theorem}
\label{Elimination_procedure}

We prove Theorem \ref{MainThm} in this section. We will use the same strategy as \cite[\S 5]{FSZ2019}, namely, an elimination procedure. So we only give an outline of this elimination procedure.

We denote by $I_s = \{ 3,5,7,\cdots,s \}$. For any subset $J \subset I_s$ with $|J| = |\mathcal{\varPsi}_B| -1$, since the following general Vandermonde matrix (see, for instance, \cite[pp. 76-77]{GK02})
\[\big[b^{j}\big]_{b \in \mathcal{\varPsi}_B, \  j \in \{1\} \cup J }\]
is invertible, there exist integers $w_b \in \mathbb{Z}$ for all $b \in \mathcal{\varPsi}_B$ such that $\sum_{b \in \mathcal{\varPsi}_B} w_{b} b^{j} = 0$ for any $j \in J$ and $\sum_{b \in \mathcal{\varPsi}_B} w_{b} b \neq 0$. (Note that these $w_b$ depend only on $J$ and $\mathcal{\varPsi}_B$.) Since
\begin{equation}\label{Hurwitz_zeta_to_zeta}
\sum_{k=1}^{b} \zeta\left(i,\frac{k}{b}\right) = \sum_{k=1}^{b} \sum_{m=0}^{\infty} \frac{b^i}{(mb+k)^i} =  b^i \zeta(i),
\end{equation}
we derive that (recall Proposition \ref{SetProposition} (2), $\frac{k}{b} \in \mathcal{F}_B$)
\[\widehat{r}_{n,b} \coloneqq \sum_{k=1}^{b} r_{n,\frac{k}{b}} = \sum_{k=1}^{b} \rho_{0,\frac{k}{b}} + \sum_{i \in I_s} \rho_{i} b^i \zeta(i)\]
is a linear combination of odd zeta values. By Lemma \ref{LemmaAna}, we have $\widehat{r}_{n,b} = (b+o(1))r_{n,1}$  as  $n \rightarrow +\infty$. Let
\[\widetilde{r}_{n} \coloneqq \sum_{b \in \mathcal{\varPsi}_B} w_b \widehat{r}_{n,b},\]
then
\begin{equation}\label{widetilde_r_n_arithmetic}
\widetilde{r}_{n} = \sum_{b \in \mathcal{\varPsi}_B} w_b \sum_{k=1}^{b} \rho_{0,\frac{k}{b}} + \sum_{i \in I_s \setminus J} \left( \sum_{b \in \mathcal{\varPsi}_B} w_b b^i \right) \rho_{i} \zeta(i),
\end{equation}
and as $n \rightarrow +\infty$,
\begin{equation}\label{widetilde_r_n_analysis}
\widetilde{r}_{n} = \left( \sum_{b \in \mathcal{\varPsi}_B} w_{b} b + o(1)\right) r_{n,1} ~\text{with}~ \sum_{b \in \mathcal{\varPsi}_B} w_{b} b \neq 0
\end{equation}
Equation \eqref{widetilde_r_n_arithmetic} shows that we can eliminate any $|\mathcal{\varPsi}_B|-1$ odd zeta values.

\smallskip
\begin{proposition}\label{FinalProp}
If $g(x_0) < e^{-(s+1)}$, then the number of irrationals in the odd zeta values set $\{ \zeta(i) \}_{i \in I_s}$ is at least $|\mathcal{\varPsi}_B|$.
\end{proposition}

\begin{proof}
We argue by contradiction. Suppose the number of irrationals in $\{ \zeta(i) \}_{i \in I_s}$ is less than $|\mathcal{\varPsi}_B|$, then we can take a subset $J \subset I_s$ with $|J| = |\mathcal{\varPsi}_B| -1$ such that $\zeta(i) \in \mathbb{Q}$ for all $I_s \setminus J$, let $A$ be the common denominator of these rational zeta values. Define $\widetilde{r}_{n}$ as above for this $J$, then by \eqref{widetilde_r_n_arithmetic} and Lemma \ref{LemmaAri}, for all $n \in P_{B,{\rm den}(r)}\mathbb{N}$, we derive that
\[Ad_{n+1}^{s+1} \widetilde{r}_{n} \in \mathbb{Z}.\]
But by \eqref{widetilde_r_n_analysis}, Lemma \ref{LemmaAna} and the hypothesis $g(x_0) < e^{-(s+1)}$, we have
\[0 < \lim_{n \rightarrow +\infty} \left| Ad_{n+1}^{s+1}\widetilde{r}_{n} \right|^{\frac{1}{n}} = e^{s+1}g(x_0) < 1, \]
this is a contradiction.
\end{proof}
\smallskip

So we seek for parameters $r,s$ and $B$ to meet the requirement $g(x_0) < e^{-(s+1)}$, and at the same time to make $|\mathcal{\varPsi}_B| \sim \frac{\zeta(2)\zeta(3)}{\zeta(6)} B$ as large as possible. By Proposition \ref{g_x_0} (2), for a fixed $r$ (such that $\frac{r^r}{(r+1)^{r+1}} < e^{-1}$), if we take $B = cs^{1/2} / \log^{1/2} s$ for some constant $c$, then $\lim_{s \rightarrow +\infty} g(x_0)^{\frac{1}{s+1}} < e^{-1}$ if and only if
\[ c < \sqrt{ \frac{4\zeta(6)}{\zeta(2)\zeta(3)} \frac{(r+1)\log(r+1)-r\log(r)-1}{2r+1}}. \]
The maximum point of the function $r \mapsto \frac{(r+1)\log(r+1)-r\log(r)-1}{2r+1}$ is
\[ r_0 = \frac{\sqrt{4e^2+1}-1}{2} \approx 2.26388, \]
with maximal value $1 - \log r_0$. The constant $c_0$ in Theorem \ref{MainThm} is designed by
\[ c_0 = \sqrt{\frac{4\zeta(2)\zeta(3)}{\zeta(6)}\left( 1 -\log r_0  \right)} . \]
This leads to the following proof:

\smallskip
\begin{proof}[Proof of Theorem \ref{MainThm}]
Given any small $\varepsilon > 0$. We first fix a rational number $r = r(\varepsilon)$ sufficiently close to $r_0$ such that
\[ \frac{c_0 - \varepsilon/10}{\zeta(2)\zeta(3)/\zeta(6)} < \sqrt{ \frac{4\zeta(6)}{\zeta(2)\zeta(3)} \frac{(r+1)\log(r+1)-r\log(r)-1}{2r+1}}. \]
Take $B = cs^{1/2}/\log^{1/2} s$ with constant $c = \frac{c_0 - \varepsilon/10}{\zeta(2)\zeta(3)/\zeta(6)}$, by Proposition \ref{g_x_0} (2) and Proposition \ref{SetProposition} (1), there exists $s_0(r,\varepsilon)$ such that for all odd integer $s \geqslant s_0(r,\varepsilon)$, we have $g(x_0) < e^{-(s+1)}$ and $|\mathcal{\varPsi}_B| > (\zeta(2)\zeta(3)/\zeta(6) - \varepsilon/10)B$. Hence, by Proposition \ref{FinalProp}, the number of irrationals among $\zeta(3),\zeta(5),\cdots, \zeta(s)$ is at least
\[ |\mathcal{\varPsi}_B| > (c_0 - \varepsilon) \frac{s^{1/2}}{\log^{1/2} s} .\]
\end{proof}
\bigskip

\section{Remarks on FSZ constructions}
\label{Remarks}

If we choose a general finite set $\mathcal{F} \subset (0,1]$ of rational numbers to be the \emph{zero set} of the auxiliary function $R(t)$, like \cite{FSZ2019} and this note, we design the factor
\begin{equation}\label{A_1_F}
 A_1(\mathcal{F})^n =  \prod_{\theta \in \mathcal{F}} {\rm den}(\theta)^{(2r+1)n}
\end{equation}
to remedy the arithmetic loss from the denominators of rational zeros. Suppose our goal is to prove that there exist $D$ irrational numbers among $\zeta(3)$, $\zeta(5)$, $\cdots$, $\zeta(s)$. In order to eliminate $D-1$ zeta values, in view of \eqref{Hurwitz_zeta_to_zeta}, we assume that there exists $D$ pairwise different positive integers $b_1, b_2 , \cdots, b_D$ such that
\begin{equation}\label{F_can_elimanate_D_values}
 \mathcal{F} \supset \left\{ \frac{1}{b_i} ,\frac{2}{b_i} , \cdots, \frac{b_i}{b_i} \right\}
\end{equation}
for any $i=1,2,\cdots,D$. Then $\mathcal{F}$ contains the following disjoint union:
\[ \mathcal{F} \supset \bigcup_{i=1}^{D} \left\{ \frac{a}{b_i} ~\bigg|~ 1 \leq a \leq b_i, \gcd(a,b_i) = 1 \right\} .\]
Hence, we have
\[ A_1(\mathcal{F}) \geqslant \prod_{i=1}^{D} b_i^{(2r+1)\varphi(b_i)}. \]
Now we consider the magnitude of $A_1(\mathcal{F})$:

\smallskip
\begin{proposition}\label{A_1_F_geq}
If $b_1,b_2,\cdots,b_D$ are $D$ pairwise distinct positive integers, then
\[ \prod_{i=1}^{D} b_i^{\varphi(b_i)} \geqslant \exp\left( \left( \frac{1}{2}\frac{\zeta(6)}{\zeta(2)\zeta(3)} +o_{D \rightarrow +\infty}(1)\right) D^2\log D   \right) . \]
\end{proposition}
\begin{proof}
We have
\begin{align*}
\log \prod_{i=1}^{D} b_i^{\varphi(b_i)} &\geqslant \sum_{i=1}^{D} \varphi(b_i)\log \varphi(b_i) \\
&\geqslant \sum_{i=1}^{D} \varphi(b'_i)\log \varphi(b'_i)
\end{align*}
where $b'_1,b'_2,\cdots,b'_D$ are the $D$ smallest positive integers in the linear order $\prec $ defined by
\[m_1 \prec m_2 \Leftrightarrow \left(  \varphi(m_1)<\varphi(m_2) \text{~or~} \left( \varphi(m_1) = \varphi(m_2) \text{~and~} m_1 < m_2 \right) \right). \]
For any positive real number $x$, we define $\mathcal{\varPsi}_x =  \{ b \in \mathbb{N} \mid \varphi(b) \leqslant x \}$. Then there exists an integer $B$ such that $\mathcal{\varPsi}_{B-1} \subset \{ b'_1,\cdots,b'_D \} \subset \mathcal{\varPsi}_{B}$. Following the same lines in the proof of Lemma \ref{A1A2}, we completes this proposition.
\end{proof}
\smallskip

Proposition \ref{A_1_F_geq} and Lemma \ref{A1A2} show that, under the constraints \eqref{A_1_F} and \eqref{F_can_elimanate_D_values}, $\mathcal{F} = \mathcal{F}_B$ in Definition \ref{SetDefinition} is the optimal choice.

For the factorial factor
\[ \frac{n!^{s+1}}{\left( \frac{n}{{\rm den}(r)} \right)!~^{{\rm den}(r)(2r+1)|\mathcal{F}|}}, \]
comparing to the corresponding factor $n!^{s+1-(2r+1)|\mathcal{F}|}$ in \cite{BR01} or \cite{FSZ2019}, we have an extra waste of
\[ \binom{n}{ \underbrace{\frac{n}{{\rm den}(r)},\cdots,\frac{n}{{\rm den}(r)}}_{{\rm den}(r) \text{~in number~}}}^{(2r+1)|\mathcal{F}|} \leqslant {\rm den}(r)^{(2r+1)|\mathcal{F}|n}, \]
which is asymptotically negligible with respect to $A_1(\mathcal{F})^n$.

There exist some arithmetic saving factors known as $\Phi_n$ factors. They are certain products over primes in the range $C_{B,r} \sqrt{n} \leqslant p \leqslant n$ (Here we can take $C_{B,r}$ as $\sqrt{2(r+1)B\log\log B}$). We mention that the saving from $\Phi_n^{-1}$ plays an important role in small cases for the odd zeta problem, see, for instance, \cite{Zu01,RZ18}, \cite[\S 4]{Zu02}, \cite[Chapitre 11]{KR07}. However, like \cite[Remark 2]{FSZ2019}, the known types of $\Phi_n$ factors have no effect on asymptotics. The reason is that, by Definition \ref{auxiliary_function}, equations \eqref{a_ik} and \eqref{R_equalsto_FFF}, for any $k \in \{ 0,1,\cdots,n \}$,
\begin{align*}
 a_{s,k} &= \left( (t+k)^s R_n(t) \right) \mid_{t=-k}  \\
 &= (-1)^{k}\binom{n}{k}^s \frac{n!(k+1)_{rn}(-k+n+1)_{rn}}{ \left( \frac{n}{ {\rm den}(r)} \right)!~^{{\rm den}(r)(2r+1)} } \prod_{\frac{a}{b} \in \mathcal{F}_B \setminus \{1\} } F_{b,a}(-k).
\end{align*}
For any prime $p$ with $C_{B,r} \sqrt{n} \leqslant p \leqslant n$, the $p$-adic order of $a_{s,0}$ is relatively small. If we define
\[ \widetilde{\Phi}_n =  \prod_{C_{B,r} \sqrt{n} \leqslant p \leqslant n~~~~} p^{v_p(\gcd\{ a_{s,k} \}_{k=0}^{n} )}, \]
then we can show that $\widetilde{\Phi}_n \leqslant A_2(B)^n \cdot d_n^{({\rm den}(r)(2r+1) + 1)|\mathcal{F}_B|}$, which is asymptotically negligible. One may want to directly save the common divisor of $d_{n+1}^{s+1}\rho_{0, \theta}$ and $d_{n+1}^{s+1}\rho_i$, but it is out of current research. The small cases are more difficult to study, up to now, except Ap{\' e}ry's theorem \cite{Ap79} that $\zeta(3)$ is irrational, the most remarkable result is that at least one of $\zeta(5),\zeta(7),\zeta(9),\zeta(11)$ is irrational, due to Zudilin \cite{Zu01} in 2001.

At last, we claim that, by making out all the implicit constants, we have a weaker but explicit result: For all odd integer $s \geqslant 10^4$, there are at least $\frac{1}{10} \frac{s^{1/2}}{(\log s)^{1/2}}$ many irrational numbers among $\zeta(3),\zeta(5),\zeta(7),\cdots,\zeta(s)$.
\bigskip

\end{document}